\documentclass[11pt,a4paper,psamsfonts]{article}

\makeatletter\let\orig@item=\@item \def\@item[#1]{\orig@item[\rm #1]}\makeatother
\let\origcaption=\caption \renewcommand\caption[1]{\parbox{0.66\textwidth}{\origcaption{#1}}}

\usepackage{amsmath,amsthm}
\usepackage{amssymb,esint}
\usepackage{amscd}
\usepackage{enumerate}
\usepackage{anysize}
\usepackage{color}

\marginsize{3.5cm}{3.5cm}{2.5cm}{2.5cm}

\usepackage[]{hyperref}
\hypersetup{
    colorlinks=true,       
    linkcolor=black,          
    citecolor=black,        
    filecolor=magenta,      
    urlcolor=cyan           
}

\newtheorem{theorem}{Theorem}[section]

\newtheorem{lemma}[theorem]{Lemma}

\newtheorem{proposition}[theorem]{Proposition}
\newtheorem{remark}[theorem]{Remark}

\DeclareMathOperator{\diff}{d}
\DeclareMathOperator{\pv}{pv}

\renewcommand\({\left(}
\renewcommand\){\right)}

\def\nc{\newcommand}
\def\be{\beta}

\def\ra{\rightarrow}
\def\la{\leftarrow}

\nc\pa{\partial}

\nc\CC{\mathbb{C}}
\nc\RR{\mathbb{R}}
\nc\QQ{\mathbb{Q}}
\nc\ZZ{\mathbb{Z}}
\nc\NN{\mathbb{N}}

\def\ba{\begin{align}}
\def\bad{\begin{aligned}}
\def\be{\begin{equation}}
\def\ea{\end{align}}
\def\ead{\end{aligned}}
\def\ee{\end{equation}}
\def\e{\eqref}

\def\dalpha{\diff \! \alpha}

\def\dt{\diff \! t}

\def\dx{\diff \! x}

\def\dxdt{\dx \dt}

\def\fract{\frac{\diff}{\dt}}

\def\defn{\mathrel{:=}}
\def\eps{\varepsilon}
\def\la{\left\vert}
\def\lA{\left\Vert}
\def\bla{\big\vert}

\def\le{\leq}
\def\les{\lesssim}
\def\mez{\frac{1}{2}}
\def\ra{\right\vert}
\def\rA{\right\Vert}
\def\bra{\big\vert}

\def\tdm{\frac{3}{2}}

\def\xN{\mathbb{N}}
\def\xR{\mathbb{R}}
\def\xS{\mathbb{S}}

\def\xZ{\mathbb{Z}}

\def\a{\alpha}

\def\XXint#1#2#3{{\setbox0=\hbox{$#1{#2#3}{\int}$ }
		\vcenter{\hbox{$#2#3$ }}\kern-.6\wd0}}

	\title{On the dynamics of the roots of polynomials under differentiation}
\author{Thomas Alazard, Omar Lazar, Quoc-Hung Nguyen}

\date\empty

\begin{document}

\maketitle

\begin{abstract}

This article is devoted to the study of a nonlinear and nonlocal parabolic equation introduced by 
Stefan Steinerberger to study the roots of polynomials under differentiation; it also appeared in a work 
by Dimitri Shlyakhtenko and Terence Tao on free convolution. 
Rafael Granero-Belinch\'on obtained a global well-posedness result for initial data small enough in a Wiener space, and recently 
Alexander Kiselev and Changhui Tan proved a global well-posedness result for any initial data in the Sobolev space $H^s(\xS)$ with $s>3/2$. 
In this paper, we consider the Cauchy problem in the critical space $H^{1/2}(\xS)$. Two interesting new features, at this level of regularity, are that 
the equation can be written in the form
$$
\partial_t u+V\partial_x u+\gamma \Lambda u=0,
$$
where $V$ is not bounded and $\gamma$ is not bounded from below. Therefore, the equation is only weakly parabolic. 
We prove that nevertheless the Cauchy problem is well posed locally in time and that the solutions are smooth for positive times. 
Combining this with the results of Kiselev and Tan, this gives a global well-posedness result for any initial data in $H^{1/2}(\xS)$. 
Our proof relies on sharp commutators estimates and introduces a strategy to prove a local well-posedness result in a situation  
where the lifespan depends on the profile of the initial data and not only on its norm.
\end{abstract}

\section{Introduction}

Stefan Steinerberger studied in \cite{Steinerberger-2018} 
the following question: 
Considering a polynomial $p_n$ of degree $n$ having all its roots 
roots on the real line distributed according to a smooth function $u_0(x)$, and a real-number $t\in (0,1)$, 
how is the distribution of the roots of the derivatives $\partial_x^k p_n$ distributed with $k=\lfloor t\cdot n\rfloor$? 
This question led him to discover a nice nonlocal nonlinear equation of the form
\be\label{0}
\partial_tu+\frac{1}{\pi}\partial_x\left( \arctan \left(\frac{Hu}{u}\right)\right)=0,
\ee
where the unknown $u=u(t,x)$ is a positive real-valued function. 

Besides its aesthetic aspect, this equation has many interesting features. 
Shlyakhtenko and Tao \cite{S-Tao-2020} derived the same equation in the context of free probability and random matrix theory 
(see also \cite{Steinerberger-3}). However, our motivation comes from the links between this equation and many models studied in fluid dynamics.

In this paper, we assume that 
the space variable $x$ belongs to the circle $\xS=\xR/(2\pi\xZ)$, and $H$ is the  
circular Hilbert transform (which acts on periodic functions), 
defined by
\be\label{11}
H g(x)=
\frac{1}{2\pi}\pv
\int_{\xS} \frac{g(x)-g(x-\alpha)}{\tan(\alpha/2)}\dalpha,
\ee
where the integral is understood as a principal value. Granero-Belinch\'on (\cite{Granero2018}) 
proved the local existence of smooth solutions 
for initial data $u_0$ in the Sobolev space $H^2(\xS)=\{u\in L^2(\xS)\,;\, \partial_{x}^2u\in L^2(\xS)\}$, 
as well as the global existence under a 
condition in appropriate Wiener spaces. 
Then, Kiselev and Tan proved \cite{Kiselev-Tan} that the Cauchy problem for \e{0} is globally well-posed in the homogeneous Sobolev space $\dot H^s(\xS)$ for all $s>3/2$, 
where $\dot H^s(\xS)=\{u\in L^2(\xS)\,;\, \Lambda^s u\in L^2(\xS)\}$ where $\Lambda$ denotes the fractional Laplacian:
$$
\Lambda=\partial_x H=(-\partial_{xx})^\mez.
$$

In fact, the equation \e{0} enters the family of fractional parabolic equations, which has attracted a lot of attention in recent years. 
To see this, introduce the coefficients
$$
V=-\frac{1}{\pi}\frac{Hu}{u^2+(Hu)^2}\quad,\quad 
\gamma=\frac{1}{\pi}\frac{u}{u^2+(Hu)^2}\cdot
$$
Then the equation~\e{0} has the following form
\begin{equation}\label{1}
 \partial_tu+V\partial_x u+\gamma \Lambda u=0.
\end{equation}
This last equation shares many characteristics 
with the Hele-Shaw equation, the Muskat equation, the 
dissipative quasi-geostrophic equation, and  1D models for the 3D incompressible Euler equation (\cite{CLM-CPAM1985}), to name a few. 
Many different techniques have been introduced to study these problems. 

For the dissipative SQG equation, the global regularity 
has been proved by Kiselev, Nazarov and Volberg~\cite{KNV-2007}, Caffarelli-Vasseur~\cite{Caffarelli-Vasseur-AoM} 
and 
Constantin-Vicol~\cite{CV-2012} (see also~\cite{KN-2009,Silvestre-2012,VaVi,NgYa}). The nonlinearity in the Muskat equation is more 
complicated. However, 
Cameron has succeeded in~\cite{Cameron} to apply 
the method introduced by Kiselev-Nazarov-Volberg to prove 
the existence of global solutions in time 
when the product of the maximum and minimum slopes is less than $1$. 
Recently, many works have extended this last result. 
The main results in \cite{AN1,AN2,AN3,AN4} imply that 
the Cauchy problem can be solved for non-Lipschitz initial 
data, following earlier work 
by Deng, Lei and Lin~\cite{DLL}, Cameron~\cite{Cameron}, 
C\'ordoba and Lazar~\cite{Cordoba-Lazar-H3/2}, Gancedo 
and Lazar~\cite{Gancedo-Lazar-H2} which allowed arbitrary slopes 
of large size. Recently, in \cite{CNX-2021} the third author proved 
local existence with $C^1$ initial data. 
The main issue is that it is difficult to dispense with the 
assumption of finite slopes. Indeed, the classical nonlinear estimates 
require to control the $L^\infty$ norm of some factors, 
which is the same for the Muskat problem 
to control the $L^\infty$ norm of the slope, which in turn 
amounts to controlling 
the Lipschitz norm of~$f$. 
Second, the Muskat equation is a degenerate parabolic equation 
for solutions which 
are not controlled in the Lipschitz norm.   We also refer
interested readers to \cite{KeHung1} for another non-local parabolic equation.

Inspired by these results, our goal here is to solve the Cauchy problem 
for \e{0} 
in the critical Sobolev space $H^\mez(\xS)$. 
Several interesting difficulties appear at that level of regularity. 

The main result of this paper is the following

\begin{theorem}\label{T1}
For all initial data $u_0$ in $\dot H^{\frac{1}{2}}(\mathbb{T})$ 
such that $\inf u_0>0$, 
the Cauchy problem has a  global in time solution $u$ 
satisfying the following properties:
\begin{enumerate}[(i)]
\item $u\in C^0([0,+\infty);\dot{H}^\mez(\xS))\cap L^2((0,+\infty);\dot{H}^1(\xS))$ 
together with the estimate
\begin{equation}\label{51}
	\sup_{t>0}\lA u(t)\rA_{\dot H^{\frac{1}{2}}}^2
	+\int_0^{\infty}\int_\xS \frac{u|\Lambda u|^2}{u^2+(Hu)^2}\dxdt
	\leq 10 \lA u_0\rA_{\dot H^{\frac{1}{2}}}^2.
\end{equation}
\item\label{i(ii)}  $u\in C^\infty((0,+\infty)\times \xS)$ and moreover 
for any $s>0$ and $\eps_0>0$ there exists a constant 
$C(s,\eps_0)$ such that
\begin{equation*}
\sup_{t>0} t^{s+\epsilon_0}\lA u(t)\rA_{\dot H^{\frac{1}{2}+s}}\le C(s,\eps_0)\lA u_0\rA_{\dot{H}^\mez}.
\end{equation*}
\item $\displaystyle\inf_{x\in\xS} u(t,x)\geq \displaystyle\inf_{x\in\xS} u_0(x)$ for $t>0$.
\end{enumerate}
\end{theorem}
\begin{remark}
The main difficulty is that 
the coercive quantity that appears in the left-hand member of \e{51}, i.e.
\be\label{co}
\int_0^{\infty}\int_\xS \frac{u|\Lambda u|^2}{u^2+(Hu)^2}\dxdt,
\ee
is insufficient to control 
the $\lA \cdot\rA_{L^2_t\dot{H}_x^1}$-norm of $u$. 
Indeed, even if we assume that the initial value of $u$ is bounded 
(which propagates in time), since the Hilbert transform is 
not bounded from $L^\infty(\xS)$ to itself, we would 
have no control from below of the denominator $u^2+(Hu)^2$.
\end{remark}
Regarding uniqueness, we have the following theorem,
\begin{theorem}\label{Tu}
For any $a>0$ and initial data $u_0 \in \dot H^{\frac{1}{2}+a}(\mathbb{S})$, 
such that $\inf u_0>0$,  the Cauchy problem has a  unique global in time solution.
\end{theorem}

\begin{remark}
Our uniqueness theorem improves the one obtained by Kiselev and Tan \cite{Kiselev-Tan}. Indeed,  they proved uniqueness in the space $\dot H^{\frac{3}{2}+a}(\mathbb{S})$ while our result deals with data in the space $\dot H^{\frac{1}{2}+a}(\mathbb{S})$. The question of uniqueness in the case $\dot H^{\frac{1}{2}}(\mathbb{S})$ remains open.
\end{remark}

\paragraph*{Plan of the paper} 
Since Kiselev and Tan (\cite{Kiselev-Tan}) proved that 
the Cauchy problem for \e{0} is globally well-posed in 
the homogeneous Sobolev space $H^s(\xS)$ for all $s>3/2$, 
it will be sufficient to prove that a local well-posedness 
result, together with the fact that the solutions are 
smooth for positive times. 
We start in Section \S\ref{S2} by discussing a model 
equation, to explain one (somewhat classical) commutator 
estimate and to explain 
the main difficulty one has to cope with. Then we prove 
the local-well posedness result in \S\ref{S3} and establish 
the smoothing effect in \S\ref{S4}.

\section{Estimate in $\dot{H}^{1/2}$ for a toy model}\label{S2}

To prove Theorem~\ref{T1}, 
the main difficulty is that, even if we know that $u$ satisfies 
a maximum principle, the coefficient $\gamma$ is not bounded from 
below by a positive constant. Indeed, it is well-known that the Hilbert 
transform is not bounded on $L^\infty(\xS)$. 
This means that \e{1} is a degenerate parabolic equation. 

Although it is not essential for the rest of the paper, it helps if we 
begin by examining a model equation with some non-degenerate dissipative term. 
Our goal here is to introduce a basic 
commutator estimate which allows to deal with equations of the form~\e{1}.  

Consider the equation
\begin{equation}\label{3}
	\partial_tu+a(u,Hu)\Lambda u=b(u,Hu)H\Lambda u,
\end{equation}
where $a$ and $b$ are two $C^\infty$ real-valued functions 
defined on $\xR^2$, satisfying
$a\ge m>0$ for some given positive constant $m$, 
together with
$$
\sup_{(x,x',y,y')\in\xR^4} \frac{\la b(x,y)-b(x',y')\ra}{\la x-x'\ra+\la y-y'\ra}<+\infty.
$$

\begin{proposition}
There exists a constant $C>0$ such that, for all $T>0$ and for 
all $u\in C^1([0,T];\dot{H}^\mez(\xS))$ solution to \e{3}, there holds
\be\label{7}
\mez\fract\lA u\rA_{\dot{H}^\mez}^2+m\int_\xS \la \Lambda u\ra^2\dx\le C\lA u\rA_{\dot{H}^\mez}
\lA \Lambda u\rA_{L^2}^2. 
\ee
\end{proposition}
\begin{remark}
Using classical arguments, it is then possible to infer 
from the {\em a priori} estimate~\e{7} a global well-posedness result for 
initial data which are small 
enough in $H^\mez(\xS)$. 
However, the study of the local well-posedness of the Cauchy problem 
for large data is more difficult and requires and extra argument 
which is explained in the next section. 
\end{remark}
\begin{proof}
Let us use the short notations 
$a=a(u,Hu)$ and $b:=b(u,Hu)$. 
Multiplying the equation~\e{3} by $\Lambda u$ and integrating over $\xS$, we obtain
\be\label{6}
\mez\fract\lA u\rA_{\dot{H}^\mez}^2+\int_\xS a\la \Lambda u\ra^2\dx
=I\defn\int_\xS b(H\Lambda u)(\Lambda u)\dx. 
\ee
To estimate $I$, we exploit the fact that 
$H^*=-H$ to write
\be\label{42}
\begin{aligned}
I&=\mez \int_\xS b(H\Lambda u)( \Lambda u)\dx
-\mez\int_\xS (\Lambda u)H \big(b\Lambda u\big)\dx\\
&=\mez \int_\xS\big(\big[ b,H\big]\Lambda u\big)\Lambda u\dx.
\end{aligned}
\ee
Now we claim that 
\be\label{7-1}
\lA \big[ b,H\big]\Lambda u\rA_{L^2}\les 
\lA b\rA_{\dot{H}^\mez}\lA \Lambda u\rA_{L^2}.
\ee
Indeed, this follows from 
the 
Sobolev embedding $\dot{H}^\mez\subset \text{BMO}$ and the classical commutator estimate
\be\label{43}
\lA [H,f]v\rA_{L^2}\lesssim \lA f\rA_{\text{BMO}}\lA v\rA_{L^2}.
\ee

(Alternatively one can prove \e{7-1} 
directly using the definition of the Hilbert transform 
as a singular integral and the Gagliardo semi-norm; see below.)  It follows that
\be\label{5}
I\les \lA b \rA_{\dot{H}^\mez}
\lA \Lambda u\rA_{L^2}^2.
\ee

Now we estimate the $\dot{H}^\mez$-norm of $b$ by means of the following elementary estimate.

\begin{lemma}\label{L:2.3}
Consider a $C^\infty$ function $\sigma\colon \xR^2\rightarrow \xR$ satisfying
$$
\forall (x,x',y,y')\in \xR^4,\qquad \la \sigma(x,y)-\sigma(x',y')\ra \le K \la x-x'\ra+K\la y-y'\ra.
$$
Then, for all $s\in (0,1)$ and all $u \in \dot{H}^{s}(\xR)$, 
one has $\sigma(u,Hu)\in \dot{H}^{s}(\xR)$ together with the estimate
\be\label{comp}
\lA \sigma(u,Hu)\rA_{\dot{H}^s}\le K\lA u\rA_{\dot{H}^s}.
\ee
\end{lemma}
\begin{proof}
By assumption, for any $\alpha\in \xR$, we have
$$
\lA \delta_\alpha \sigma(u,Hu)\rA_{L^2} \le K \lA \delta_\alpha u\rA_{L^2}
+\lA \delta_\alpha Hu\rA_{L^2}.
$$
Then by using the Gagliardo semi-norms, we get 
$$
\lA \sigma(u,Hu)\rA_{\dot{H}^s}\le K\lA u\rA_{\dot{H}^s}
+K\lA Hu\rA_{\dot{H}^s},
$$
and the desired result follows since 
$\lA Hu\rA_{\dot{H}^s}=\lA u\rA_{\dot{H}^s}$.
\end{proof}

The previous lemma implies that 
$$
\lA b\rA_{\dot{H}^\mez}\les \lA u\rA_{\dot{H}^\mez}+\lA Hu\rA_{\dot{H}^\mez}\les \lA u\rA_{\dot{H}^\mez},
$$
and we deduce from~\e{5} that
$$
I\les \lA u\rA_{\dot{H}^\mez}
\lA \Lambda u\rA_{L^2}^2.
$$
Therefore the wanted result \e{7} follows from~\e{6}.
\end{proof}

\section{Local well-posedness}\label{S3}

We construct solutions to \e{0} 
as limits of solutions to a sequence of 
approximate nonlinear systems. 
We divide the analysis into three parts.
\begin{enumerate}
\item We start by proving that the Cauchy problem for these systems 
systems are well posed globally 
in time and satisfy the maximum principles. In particular, 
the approximate solutions are bounded by a positive constant. 
\item Then, we show that  
the solutions of the approximate systems are 
bounded in $C^0([0,T];\dot{H}^\mez(\xS))$. 
on a uniform time interval that depends on the profile of the initial 
data (and not only on their norm).
\item The third task is to show that these approximate solutions converge to a limit which is 
a solution of the original equation. 
To do this, we use interpolation and compactness arguments.
\end{enumerate}

\subsection{Approximate systems}\label{S:approximate}

Fix $\delta\in (0,1]$ and 
consider the following approximate Cauchy problem:
\begin{equation}\label{A3}
\left\{
\begin{aligned}
&\partial_tu+\frac{1}{\pi}\frac{u\Lambda u-(Hu)\partial_x u}{\delta+u^2+(Hu)^2}-\delta \partial_x^2u=0,\\
&u\arrowvert_{t=0}=e^{\delta \partial_x^2} u_0.
\end{aligned}
\right.
\end{equation}

The following lemma states that this Cauchy problem 
has smooth solutions.

\begin{lemma}
For any positive initial data $u_0\in L^2(\xS)$ 
and for any $\delta>0$, the initial value problem~\e{A3} 
has a unique solution 
$u$ in $C^{1}([0,+\infty);H^{\infty}(\xS))$. 
This solution is such that, for all $t\in [0,+\infty)$,
\be\label{12}
\inf_{x\in\xS} u(t,x)\geq \inf_{x\in\xS} u_0(x)\quad\text{and}\quad \max_{x\in\xS}u(t,x)\le \max_{x\in\xS}u_0(x).
\ee
\end{lemma}
\begin{proof}
The proof is classical and follows from arguments already introduced by Granero in~\cite{Granero2018}, but we repeat it for completeness. 

Fix $\delta>0$. 
The Cauchy problem \eqref{A3} has the following form
\begin{equation}\label{edo}
\partial_t u-\delta\partial_x^2 u= F_\delta,\quad u\arrowvert_{t=0}=e^{\delta \partial_x^2} u_0,
\end{equation}
where 
$$
F_\delta=-\frac{1}{\pi}\frac{u\Lambda u-(Hu)\partial_x u}{\delta+u^2+(Hu)^2}\cdot
$$

{\em Step 1: existence of mild solution locally in time}. 
Since $H^{\mez+\nu}(\xS)\subset L^\infty(\xS)$ 
for all $\nu>0$ and since the circular Hilbert transform 
$H$ is bounded on 
$H^{\mez+\nu}(\xS)$, 
it is easy to verify that
\begin{align*}
\lA F_\delta(u)-F_\delta(v)\rA_{L^2}&\les_{\delta} \big(\lA u\rA_{H^{\mez+\nu}}
+\lA v\rA_{H^{\mez+\nu}}\big)\lA u-v\rA_{H^1}\\
&\quad+\big(\lA u\rA_{H^1}
+\lA v\rA_{H^1}\big)\lA u-v\rA_{H^{\mez+\nu}},
\end{align*}
where the notation $\les_{\delta}$ is intended to indicate that the 
implicit constant depends on $\delta$. 
Therefore, it follows from the interpolation inequality in Sobolev space, from the fixed point theorem and from 
the usual energy estimate for the heat equation that 
the Cauchy problem has a unique mild 
solution in $C^{0}([0,T_\delta);L^2(\xS))\cap L^2(0,T_\delta;H^1(\xS))$, where 
$T_\delta$ is estimable from below in terms of $\lA u\arrowvert_{t=0}\rA_{L^2}$ (see \cite[Section 15.1]{Taylor3}). 
In particular, we have the following alternative: either
\be\label{A4}
T_\delta=+\infty\qquad\text{or}\qquad \limsup_{t\rightarrow T_\delta} \lA u(t)\rA_{L^2}=+\infty.
\ee

{\em Step 2 : Global well-posedness} 
On the other hand, directly from the obvious estimate
\begin{align*}
\lA F_\delta\rA_{L^2}\les_{\delta} \lA u\rA_{H^{1}},
\end{align*}
the energy estimate for the heat equation implies that 
$\limsup_{{t\rightarrow T_n}} \lA u(t)\rA_{L^2}=+\infty$ 
is impossible with $T_n<+\infty$. This proves that the solution exists globally in time. 

{\em Step 3: Regularity.} We verify that the solution defined above 
is regular by noting that one can solve the Cauchy problem in $H^s(\xS)$ 
for all $s>3/2$ using the usual nonlinear estimates in Sobolev 
spaces and the argument above. By uniqueness, this implies that 
the mild solution defined above is continuous in time with values in $H^s(\xS)$ for all $s$.

{\em Step 4: maximum principle.} 
The claim \e{12} 
follows from the classical arguments. Firstly, notice that
$$
\inf_{x\in\xS}e^{\delta\partial_x^2}u_0(x)\ge \inf_{x\in\xS}u_0(x).
$$
On the other hand, at a point $x_t$ where the function $u(t,\cdot)$ reaches its minimum, we have
$$
\partial_x u(t,x_t)=0,\quad \partial_x^2 u(t,x_t)\ge 0, \quad 
\Lambda u(t,x_t)\le 0,
$$
where the last inequality follows from the fact that
$$
\Lambda u(x)
=\frac{1}{4\pi}\pv\int_{\xS}\frac{u(x)-u(x-\a)}{\sin(\alpha/2)^2}\dalpha.
$$
It follows that $\inf_{x\in\xS} u(t,x)\ge \inf_{x\in\xS}u_0(x)$. 

By similar arguments, we obtain the second inequality 
$\sup_{x\in\xS} u(t,x)\le \sup_{x\in\xS}u_0(x)$.
This completes the proof.
\end{proof}

\subsection{Uniform estimates}

Fix $\delta>0$ and $c_0$ and consider an initial data 
$u_0$ in $H^\mez(\xS)$ with $\inf_{x\in\xS}u_0(x)\ge c_0$. 
As we have seen in the previous paragraph, there exists a unique 
function 
$u\in C^1([0,+\infty);H^\infty(\xS))$ satisfying 
\begin{equation}\label{2}
\left\{
\begin{aligned}
&\partial_tu
+\frac{1}{\pi}\frac{u\Lambda u-(Hu)\partial_x u}{\delta+u^2+(Hu)^2}-\delta \partial_x^2u=0,\\
&u\arrowvert_{t=0}=e^{\delta \partial_x^2} u_0,\\
&\inf_{x\in\xS}u(t,x)\ge c_0.
\end{aligned}
\right.
\end{equation}

We shall prove estimates which are uniform 
with respect to $\delta\in (0,1]$ (this is why 
we are writing 
simply $u$ instead of $u_\delta$, to simplify notations). 

Let $\eps>0$. We want to estimate
$$
v\defn u-e^{(\eps+\delta)\partial_x^2}u_0.
$$
Set 
$$
u_{0,\eps}=e^{(\eps+\delta)\partial_x^2}u_0,
$$
and introduce the coefficients 
\begin{align*}
\gamma&=\frac{1}{\pi}\frac{u}{\delta+u^2+(Hu)^2},\\	V&=-\frac{1}{\pi}\frac{Hu}{\delta+u^2+(Hu)^2},\\
\rho&=\sqrt{\delta+u^2+(Hu)^2}.
\end{align*}
With the previous notations, we have
\be\label{4}
\partial_tv+V\partial_x v+\gamma\Lambda v-\delta\partial_x^2 v
=R_\eps(u,u_{0})
\ee
where 
\begin{equation*}
	R_\eps(u,u_{0})=-\gamma\Lambda u_{0,\eps}-V\partial_x u_{0,\eps}+\delta \partial_x^2 u_{0,\eps}.
\end{equation*}

\begin{lemma}\label{L:3.2}
For any $u_0\in H^\mez(\xS)$ with $\inf_{x\in\xS}u_0(x)>0$, 
there exist a constant $\varepsilon_0$ and a  
function $T\colon (0,1]\to (0,1)$ with
$$
\lim_{\eps\to 0}T(\eps)=0,
$$
such that the following result holds: 
for all $\delta\in (0,1]$, all $u\in C^1([0,+\infty);H^\infty(\xS))$ 
satisfying~\e{2} with initial data $u\arrowvert_{t=0}=e^{\delta\partial_x^2}u_0$, 
and for all $\varepsilon\in (0,\varepsilon_0]$, the function 
$v=u-e^{(\eps+\delta)\partial_x^2}u_0$ satisfies
\begin{align}\label{z1}
\sup_{t\in [0,T(\eps)]} \lA v(t)\rA_{\dot H^{\frac{1}{2}}}^2
+\int_{0}^{T(\eps)}\int_{\xS}\frac{u\la\Lambda v\ra^2}{\delta+u^2+(Hu)^2}\dx\dt
+\delta \int_{0}^{T(\eps)} \lA v\rA_{\dot{H}^\tdm}^2\dt \leq
\mathcal{F}(T(\eps)),
\end{align}
for some function $\mathcal{F}\colon\xR_+\to\xR_+$ with $\lim_{\tau\to 0}\mathcal{F}(\tau)=0$.
\end{lemma}
\begin{proof}Hereafter, $C$ denotes various constants which depend only on the constant $c_0$ (remembering that $c_0$ is some given constant such that $\inf u(t,x)\ge \inf u_0\ge c_0$) and we use the notation $A\les_{c_0} B$ to indicate that $A\le CB $ for such a constant $C$.

Consider a parameter $\kappa\in (0,1]$ whose value is to be determined. Then decompose the Hilbert transform as 
$H=H_{\kappa,1}+H_{\kappa,2}$ where
\begin{align*}
	H_{\kappa,1}g(x)&=\frac{1}{2\pi}\int_\xS  g(x-\a)\chi\(\frac{\alpha}{\kappa}\)\frac{\dalpha}{\tan(\alpha/2)},\\
	H_{\kappa,2}g(x)&=\frac{1}{2\pi}\int_\xS  g(x-\a)\(1-\chi\(\frac{\alpha}{\kappa}\)\)
	\frac{\dalpha}{\tan(\alpha/2)},
\end{align*}
for some cut-off function $\chi\in C^\infty$ satisfying 
$\chi=1$ in $[-1,1]$ and $\chi=0$ in $\mathbb{R}\setminus[-2,2]$.

Multiply equation \e{4} by $\Lambda v$ and then integrate over $\xS$, 
to obtain
\be\label{9}
\mez{\fract}\lA v\rA_{{\dot{H}^\mez}}^2+
\int_\xS \gamma(\Lambda v)^2\dx =A+B+R
\ee
where
\begin{align*}
A&=\int_\xS V(H_{\kappa,1}\Lambda v)( \Lambda v)\dx,\\
B&=\int_\xS V(H_{\kappa,2}\Lambda v)( \Lambda v)\dx,\\
R&=\int_\xS R(u,u_{0,\eps})( \Lambda v)\dx.
\end{align*}
Set
$$
W\defn\sqrt{\gamma}\Lambda v,
$$
so that the dissipative term in \e{9} is of the form
$$
\int_\xS \gamma(\Lambda v)^2\dx=\int_\xS W^2\dx.
$$

{\em Step 1: estimate of $B$ and $R$.} 
Directly from the definition of 
$\gamma$ and $V$, we have
\be\label{21}
\gamma\le \frac{1}{\pi c_0}\quad 
,\quad \la V\ra\le \frac{1}{2\pi c_0}.
\ee
One important feature of the critical problem is that the dissipative term is degenerate. This means that the coefficient $\gamma$ is not bounded from below by a fixed positive constant. 
As a result, we do not control the $L^2$-norm of 
$\Lambda v$. 
Instead, we merely control the $L^2$-norm of $W=\sqrt{\gamma}\Lambda v$. Therefore, we will systematically write $\Lambda v$ under the form
$$
\Lambda v=\frac{1}{\sqrt{\gamma}}\sqrt{\gamma}\Lambda v=\frac{1}{\sqrt{\gamma}}W.
$$

To absorb the contribution of the factor $1/\sqrt{\gamma}$ in the estimates for $B$ and $R$, it will be sufficient to notice that we have the pointwise bound
$$
\la \frac{V}{\sqrt{\gamma}}\ra=\frac{1}{u\sqrt{\delta +u^2+(Hu)^2}}\la Hu\ra
\les_{c_0}\la Hu\ra.
$$

In particular, remembering that the Hilbert transform is bounded from $L^p(\xS)$ to $L^p(\xS)$ for any $p\in (1,+\infty)$ and using the Sobolev embedding $H^s(\xS)\subset L^{2/(1-2s)}(\xS)$, we deduce that
\be\label{41}
\lA V/\sqrt{\gamma}\rA_{L^4}\les_{c_0}\lA Hu\rA_{L^4}\les_{c_0} \lA u\rA_{L^4}\les_{c_0} \lA u\rA_{\dot{H}^\mez}.
\ee
Then it follows 
from H\"older's inequality that
\begin{align*}
\la B\ra&\le \int_\xS \frac{V}{\sqrt{\gamma}}\la H_{\kappa,2}\Lambda v \ra 
\la \sqrt{\gamma}  \Lambda v\ra\dx\les_{c_0} \lA u\rA_{\dot{H}^\mez}\lA H_{\kappa,2}\Lambda v \rA_{L^4}\lA W\rA_{L^2}.
\end{align*}
On the other hand, 
$$
\lA H_{\kappa,2}\Lambda v \rA_{L^4}
=\lA H_{\kappa,2}\partial_x H v \rA_{L^4}\les\kappa^{-1}
\lA Hv\rA_{L^4}\les \kappa^{-1}
\lA v\rA_{L^4},
$$
where we used the definition of $H_{\kappa,2}$, noting that $g_x(x-\a)=\partial_\a(g(x)-g(x-a))$ and integrating by parts in $\a$.

By combining the previous estimates, we conclude that
$$
\la B\ra
\les_{c_0} \kappa^{-1}\lA u\rA_{\dot{H}^\mez}
\lA v \rA_{\dot{H}^\mez}\lA W\rA_{L^2}.
$$

The estimate of $R$ is similar. Recall that
\begin{equation*}
	R_\eps(u,u_{0})=-\gamma\Lambda u_{0,\eps}-V\partial_x u_{0,\eps}+\delta \partial_x^2 u_{0,\eps}.
\end{equation*}
To estimate the contribution of the first term, we write
$$
\la \int_\xS \gamma(\Lambda u_{0,\eps})( \Lambda v)\dx\ra 
\le \lA \sqrt{\gamma}\rA_{L^\infty}\lA \Lambda u_{0,\eps}\rA_{L^2}
\lA \sqrt{\gamma}\Lambda v\rA_{L^2}
\les_{c_0}\eps^{-\frac{1}{2}}\lA u_0\rA_{\dot{H}^\mez}\lA W\rA_{L^2},
$$
where we have used the elementary inequality
$$
\lA u_{0,\eps}\rA_{\dot{H}^1}\les (\eps+\delta)^{-\frac{1}{4}}
\lA u_{0,\eps}\rA_{\dot{H}^\mez}\les 
\eps^{-\frac{1}{4}}
\lA u_{0}\rA_{\dot{H}^\mez} ,
$$
since the Fourier transform of $u_{0,\eps}=e^{(\eps+\delta)\partial_x^2}u_0$ is essentially localized in the interval $\la \xi\ra\les \sqrt{\eps+\delta}$.
With regards to the second term, we use again the 
estimate \e{41} to get
\begin{align*}
\la \int_\xS V(\partial_x u_{0,\eps})( \Lambda v)\dx\ra 
&\le \lA V/\sqrt{\gamma}\rA_{L^4}\lA \partial_x u_{0,\eps}\rA_{L^4}
\lA \sqrt{\gamma}\Lambda v\rA_{L^2}\\
&\les_{c_0}\eps^{-1}\lA u\rA_{\dot{H}^\mez}\lA u_0\rA_{\dot{H}^\mez}\lA W\rA_{L^2}.
\end{align*}
Eventually, 
we have
$$
\lA \delta \partial_x^2 u_{0,\eps}\rA_{\dot{H}^\mez}=
\lA \delta \partial_x^2 e^{(\eps+\delta)\partial_x^2}u_{0}\rA_{\dot{H}^\mez}\les 
\lA \delta \partial_x^2 e^{\delta\partial_x^2}u_{0}\rA_{\dot{H}^\mez}
\les \lA u_0\rA_{\dot{H}^\mez}.
$$

So, by combining the previous inequalities, we conclude that
$$
B+R
\les_{c_0}\kappa^{-1}\lA u\rA_{\dot{H}^\mez}\lA v\rA_{\dot{H}^\mez}\lA W\rA_{L^2}+\eps^{-\mez}(1+\lA u\rA_{\dot{H}^\mez})
\lA u_{0}\rA_{\dot{H}^\mez}\lA W\rA_{L^2}+\lA u_0\rA_{\dot{H}^\mez},
$$
hence, replacing $u$ by $v+u_{0,\eps}$ in the right-hand side, we conclude that
\be\label{B+R}
\begin{aligned}
B+R& \les_{c_0}\left(\kappa^{-1}\lA u\rA_{\dot{H}^\mez}+\eps^{-\mez}
\lA u_{0}\rA_{\dot{H}^\mez}\right)\lA v\rA_{\dot{H}^\mez}\lA W\rA_{L^2}\\
&\qquad+\eps^{-\mez}
\big(\lA u_{0}\rA_{\dot{H}^\mez}+\lA u_{0}\rA_{\dot{H}^\mez}^2\big)\lA W\rA_{L^2}+\lA u_0\rA_{\dot{H}^\mez}.
\end{aligned}
\ee

{\em Step 2: Estimate of $A$.} By an argument 
parallel to \e{42}, the fact that 
$H_{\kappa,1}^*=-H_{\kappa,1}$ implies that
\begin{align*}
A&=\mez \int_\xS V(H_{\kappa,1}\Lambda v)( \Lambda v)\dx-\mez\int_\xS (\Lambda v)H_{\kappa,1}( V\Lambda v)\dx\\
&=\mez \int_\xS\big(\big[ V,H_{\kappa,1}\big]\Lambda v\big)\Lambda v\dx.
\end{align*}
Then, the Cauchy-Schwarz inequality implies that
\be\label{A}
A\le \left(\int_\xS \gamma^{-1}\bla 
[V,H_{\kappa,1}](\Lambda v)\bra^2\dx\right)^{\frac{1}{2}} \lA \sqrt{\gamma}\Lambda v\rA_{L^2}.
\ee

We now have to estimate the commutator $[V,H_{\kappa,1}]$. 
Previously in \S\ref{S2}, we deduced the 
commutator estimate~\e{7} from~\e{43}. This time, we will proceed directly from the definition of $H_{\kappa,1}$, without a d\'etour by $\text{BMO}$. The main new point is that this will allow us to obtain an estimate in terms of $\lA W\rA_{L^2}$ instead of $\lA \Lambda v\rA_{L^2}$. Namely, directly from the definition of $H_{\kappa,1}$, we have
\begin{align*}
[V,H_{\kappa,1}](\Lambda v)&=
\frac{1}{4\pi}\pv
\int_{\xS} \frac{V(x)(\delta_\a \Lambda v)(x)-\delta_\a (V \Lambda v)(x)}{\tan(\alpha/2)}\chi\(\frac{\alpha}{\kappa}\)\dalpha\\
&=-\frac{1}{4\pi}\pv 
\int_{\xS} \frac{(\delta_\a V)(x)(\Lambda v)(x-\a)}{\tan(\alpha/2)}\chi\(\frac{\alpha}{\kappa}\)\dalpha\\
&=-\frac{1}{4\pi}\pv 
\int_{\xS} \frac{(\delta_\a V)(x)}{\sqrt{\gamma}(x-\a)}W(x-\a)\chi\(\frac{\alpha}{\kappa}\)\frac{\dalpha}{\tan(\alpha/2)}
\end{align*}
where we replaced $\Lambda v$ by $W/\sqrt{\gamma}$ to obtain the last identity.

Therefore,
\begin{align*}
&\int_\xS \gamma^{-1}\bla 
[V,H_{\kappa,1}](\Lambda v)\bra^2\dx\\
&\qquad\les  
\int_\xS(\gamma(x))^{-1}
\left(\int_{|\alpha|\leq 2\kappa} \la \delta_\alpha  V(x)\ra \gamma(x-\alpha)^{-1/2}\la W(x-\alpha)\ra\frac{\dalpha}{|\tan(\alpha/2)|}\right)^2 \dx\\
&\qquad\les_{c_0}
\int_\xS(\rho^2(x)
\left(\int_{|\alpha|\leq 2\kappa} \la \delta_\alpha  V(x)\ra \rho(x-\alpha)\la W(x-\alpha)\ra\frac{\dalpha}{|\tan(\alpha/2)|}\right)^2 \dx.
\end{align*}
\begin{lemma}
Introduce the notation
$$
Q_\alpha(g)\defn |\delta_\alpha g|+|\delta_\alpha Hg|.
$$
Then there holds
\begin{align}\label{e1}
\la\delta_\alpha V\ra \lesssim_{c_0} \frac{Q_\alpha(v)(1+Q_\alpha(v))^2}{
\rho^2}+\eps^{-3}(1+\lA u_0\rA_{H^\mez})^3\la \alpha\ra.
\end{align}
\end{lemma}
\begin{proof} One has
	\begin{align*}
		|\delta_\alpha V(x)|&\lesssim \frac{|\delta_\alpha Hu(x)|}{\rho^2(x)}+|Hu(x-\alpha)|\frac{|\delta_\alpha u(x)|(|u(x)|+|u(x-\alpha)|)}
		{\rho^2(x-\alpha)\rho^2(x)}
		\\
		& \les_{c_0} \frac{|\delta_\alpha Hu(x)|}{\rho^2(x)}+\frac{|\delta_\alpha u(x)|^2+|\delta_\alpha H u(x)|^2}{\rho(x-\alpha)\rho(x)}+\frac{|\delta_\alpha u(x)|+|\delta_\alpha H u(x)|}{\rho(x-\alpha)\rho(x)}.
	\end{align*}
Since
\begin{align*}
	\frac{1}{\rho(x-\alpha)}\lesssim \frac{1}{\rho(x)}(1+|\delta_\alpha u(x)|+|\delta_\alpha Hu(x)|),
\end{align*}
we obtain 
$$
\la\delta_\alpha V\ra \lesssim_{c_0} \frac{Q_\alpha(v)(1+Q_\alpha(v))^2}{
\rho^2}\cdot
$$
To get the wanted result \eqref{e1} from this, we 
replace $u$ by $v+u_{0,\eps}$ and use the two 
following elementary ingredients:
\begin{align*}
&\frac{Q_\alpha(u_{0,\eps})(1+Q_\alpha(u_{0,\eps}))^2}{
\rho^2}\les_{c_0}Q_\alpha(u_{0,\eps})(1+Q_\alpha(u_{0,\eps}))^2
\\&Q_\alpha(u_{0,\eps})\le 
\big(\lA \partial_xu_{0,\eps}\rA_{L^\infty}
+\lA \partial_x (Hu_{0,\eps})\rA_{L^\infty} \big)\la \a\ra\les 
\lA u_{0,\eps}\rA_{H^{2}}\la \a\ra\les 
\eps^{-\frac{3}{4}}\lA u_{0}\rA_{H^\mez}\la \a\ra.
\end{align*}
Since $\la \alpha\ra^3\les \la\alpha\ra$, 
this completes the proof.
\end{proof}

Set $K(\eps)\defn 
\eps^{-3}(1+\lA u_{0,\eps}\rA_{H^\mez})^3$. 
It follows from the previous lemma and the preceding inequality that
\begin{align*}
&\int_\xS \gamma^{-1}\bla 
[V,H_{\kappa,1}](\Lambda v)\bra^2\dx
\lesssim_{c_0} (I)+(II),
\end{align*}
where
\begin{align*}
&(I)\defn 
\underset{\xS}{\int}\rho(x)^2\bigg(\underset{|\alpha|\leq 2\kappa}{\int}  \frac{Q_\alpha(v)(x)(1+Q_\alpha(v)(x))^2}{\rho(x)^2}
\rho(x-\alpha)\la W(x-\alpha)\ra
\frac{\dalpha}{|\tan(\frac{\alpha}{2})|}\bigg)^2\dx,\\
&(II)\defn K(\varepsilon)^2\underset{\xS}{\int} \rho(x)^2\bigg(\underset{|\alpha|\leq 2\kappa}{\int}  \rho(x-\alpha)\la W(x-\alpha)\ra \dalpha\bigg)^2\dx.
\end{align*}
Using the Cauchy-Schwarz inequality, we see that
$$
(I)\les \bigg(\iint_{\xS^2} Q_\alpha(v)(x)^2(1+Q_\alpha(v)(x))^4\frac{\rho^2(x-\alpha)}{\rho^2(x)} \frac{\dalpha\dx}{|\alpha|^2}\bigg)\lA W\rA_{L^2}^2.
$$
Since
\begin{align*}
	\frac{\rho^2(x-\alpha)}{\rho^2(x)}\lesssim 1+|Q_\alpha (u)(x)|^2,
\end{align*}
we end up with
$$
(I)\les \bigg(\iint_{\xS^2} Q_\alpha(v)(x)^2(1+Q_\alpha(v)(x))^4(1+Q_\alpha(u)(x))^2 \frac{\dalpha\dx}{|\alpha|^2}\bigg)\lA W\rA_{L^2}^2.
$$
On the other hand, 
\begin{align*}
 &\iint_{\xS^2} Q_\alpha(v)(x)^2(1+Q_\alpha(v)(x))^4(1+Q_\alpha(u)(x))^2 \frac{\dalpha\dx}{|\alpha|^2}\\&\lesssim
 \iint_{\xS^2} Q_\alpha(v)(x)^2(1+Q_\alpha(v)(x))^4 \frac{\dalpha\dx}{|\alpha|^2}\\&+\left(\iint_{\xS^2} Q_\alpha(v)(x)^4(1+Q_\alpha(v)(x))^8 \frac{\dalpha\dx}{|\alpha|^2}\right)^{\mez}\left(\iint_{\xS^2} Q_\alpha(u)(x)^4 \frac{\dalpha\dx}{|\alpha|^2}\right)^{\mez}\\&\lesssim  \lA v\rA_{\dot H^{\mez}}^2
\big(1+\lA v\rA_{\dot H^{\mez}}\big)^4\big(1+\lA u\rA_{\dot H^{\mez}})^2,
\end{align*}
where we used the fact that $Q_\alpha(f)(x)=|\delta_\alpha f(x)|+|\delta_\alpha Hf(x)|$ and 
\begin{align*}
\iint_{\xS^2} \Big(|\delta_\alpha f(x)|^{2\gamma}+|\delta_\alpha Hf(x)|^{2\gamma}\Big) \frac{\dalpha\dx}{|\alpha|^2}\lesssim ||f||_{\dot H^{\frac{1}{2}}}^{2\gamma}
\end{align*}
for any $\gamma\geq 1$.

This gives
$$
(I)\les \lA v\rA_{\dot H^{\mez}}^2
\big(1+\lA v\rA_{\dot H^{\mez}}\big)^4\big(1+\lA u\rA_{\dot H^{\mez}})^2\lA W\rA_{L^2}^2.
$$

On the other hand, 
\begin{align*}
(II)\les K(\varepsilon)^2\kappa^{\frac{1}{2}}\lA u\rA_{\dot H^{\mez}}^4 \lA W\rA_{L^{2}}^2
\end{align*}

Therefore, it follows from \e{A} that
\begin{align*}
A&\le \left(\int_\xS \gamma^{-1}\bla 
[V,H_{\kappa,1}](\Lambda v)\bra^2\dx\right)^{\frac{1}{2}} \lA W\rA_{L^2}\\
&\les \big( (I)+(II)\big)^\mez \lA W\rA_{L^2}\\
&\les 
\lA v\rA_{\dot H^{\mez}}
\big(1+\lA v\rA_{\dot H^{\mez}}\big)^2\big(1+\lA u\rA_{\dot H^{\mez}})\lA W\rA_{L^2}^2
+K(\varepsilon)\kappa^{\frac{1}{4}}\lA u\rA_{\dot H^{\mez}}^2 \lA W\rA_{L^{2}}^2.
\end{align*}

By combining this with \e{B+R}, we get from \e{9} that there exists a constant $C$ depending only on $c_0$ such that
\be\label{v-1}
\begin{aligned}
\mez{\fract}\lA v\rA_{{\dot{H}^\mez}}^2+
\lA W\rA_{L^2}^2
&\le 
C\lA v\rA_{\dot H^{\mez}}
\big(1+\lA v\rA_{\dot H^{\mez}}\big)^2\big(1+\lA u\rA_{\dot H^{\mez}})\lA W\rA_{L^2}^2\\
&\quad+CK(\varepsilon)\kappa^{\frac{1}{4}}\lA u\rA_{\dot H^{\mez}}^2 \lA W\rA_{L^{2}}^2\\
&\quad +C\left(\kappa^{-1}\lA u\rA_{\dot{H}^\mez}
+C\eps^{-\mez}
\lA u_{0}\rA_{\dot{H}^\mez}\right)\lA v\rA_{\dot{H}^\mez}\lA W\rA_{L^2}\\
&\quad +C\eps^{-\mez}
\Big(\lA u_{0}\rA_{\dot{H}^\mez}+\lA u_{0}\rA_{\dot{H}^\mez}^2\Big)\lA W\rA_{L^2}.
\end{aligned}
\ee

Using the Young's inequality, 
this immediately implies that an inequality of the form
\be\label{v-2}
\mez{\fract}\lA v\rA_{{\dot{H}^\mez}}^2+
\Upsilon\lA W\rA_{L^2}^2
\le M\lA v\rA_{\dot{H}^\mez}^2
+F,
\ee
where
\begin{align*}
M&\defn 4 C^2\eps^{-1}\Big(\lA u_{0}\rA_{\dot{H}^\mez}+\lA u_{0}\rA_{\dot{H}^\mez}^2\Big)^2,\\
F&\defn 4C^2\eps^{-1}\Big(\lA u_{0}\rA_{\dot{H}^\mez}+\lA u_{0}\rA_{\dot{H}^\mez}^2\Big)^2,\\
\Upsilon&\defn 
\frac{1}{4} -C\lA v\rA_{\dot H^{\mez}}
\big(1+\lA v\rA_{\dot H^{\mez}}\big)^2\big(1+\lA u\rA_{\dot H^{\mez}})\\
&\quad-CK(\varepsilon)\kappa^{\frac{1}{4}}\lA u\rA_{\dot H^{\mez}}^2 \\
&\quad-C^2\left(\kappa^{-1}\lA u\rA_{\dot{H}^\mez}+\eps^{-\mez}
\lA u_{0}\rA_{\dot{H}^\mez}\right)^2\lA v\rA_{\dot{H}^\mez}^2.
\end{align*}
In particular, as long as $\Upsilon\ge 0$, we have
$$
\lA v(t)\rA_{\dot{H}^\mez}^2\le e^{2Mt}\lA v(0)\rA_{\dot{H}^\mez}^2+\frac{e^{tM}-1}{M}F.
$$
If one further assumes that $tM\le 1$, it follows that
$$
\lA v(t)\rA_{\dot{H}^\mez}^2\le e^{2Mt}\lA v(0)\rA_{\dot{H}^\mez}^2+tF.
$$

Introduce the parameter 
$$
\nu(\varepsilon)\defn 2\lA v\arrowvert_{t=0}\rA_{\dot{H}^\mez(\xS)}= 2\lA u_0-e^{\eps\partial_x^2}u_{0}\rA_{\dot{H}^\mez(\xS)}.
$$
Then choose $\eps$ small enough, so that
$$
C\nu (\eps)
\big(1+\nu(\eps)\big)^2\big(1+2\lA u_0\rA_{\dot H^{\mez}})\le \frac{1}{16}.
$$
We 
then fix $\kappa$ small enough to that
$$
CK(\varepsilon)\kappa^{\frac{1}{4}}(2\lA u_0\rA_{\dot H^{\mez}})^2
\le \frac{1}{16},
$$
where recall that $K(\eps)\defn \eps^{-3}(1+\lA u_{0,\eps}\rA_{H^\mez})^3$.

We then deduce the wanted uniform estimate by an elementary continuation argument.
\end{proof}

\subsection{Compactness}\label{S:compactness}

Previously, we have proved {\em a priori} estimates  for the spatial derivatives. In this paragraph, we collect results from which we will derive estimates for the time derivative as well as for the nonlinearity. These estimates are used to pass to the limit in 
the equation.

Recall the notations introduced in the previous section, as well as the estimates proved there. Fix $c_0>0$. Given $\delta\in (0,1]$ and an initial data $u_0\in H^\mez(\xS)$ satisfying $u_0\ge c_0$, we have seen  that there exists a (global in time) solution $u_\delta$ to the Cauchy problem:
\begin{equation}\label{A3b}
\left\{
\begin{aligned}
&\partial_tu_\delta+\frac{1}{\pi}\frac{u_\delta\Lambda u_\delta-(Hu_\delta)\partial_x u_\delta}{\delta+u_\delta^2+(Hu_\delta)^2}-\delta \partial_x^2u_\delta=0,\\
&u_\delta\arrowvert_{t=0}=e^{\delta \partial_x^2} u_0.
\end{aligned}
\right.
\end{equation}
Moreover, we have proved that one can fix $\eps$ small enough 
such that one can write $u_\delta$ under the form
$$
u_\delta(t,x)=(e^{(\eps+\delta)\partial_x^2}u_0)(x)
+v_{\delta}(x),
$$
and there exist $T>0$ and $M>0$ depending on $u_0$ such that, for all $\delta\in (0,1]$,
\begin{align}\label{z1b}
\sup_{t\in [0,T]} \lA v_{\delta}(t)\rA_{\dot H^{\frac{1}{2}}}^2
+\int_{0}^{T}\int_{\xS}\frac{u_\delta\la\Lambda v_{\delta}\ra^2}{\delta+u_\delta^2+(Hu_\delta)^2}\dx\dt+\delta \int_{0}^{T}
\lA v_\delta\rA_{\dot{H}^\tdm}^2\dt\le M.
\end{align}

Now, to pass to the limit in the equation \e{A3b}, we need to extract some uniform estimates for the time derivative. Since $\partial_tu_\delta=\partial_tv_\delta$, it is sufficient to estimate the latter quantity. 
It is given by
\be\label{4b}
\partial_tv_\delta=-V_\delta\partial_x v_\delta
-\gamma_\delta\Lambda v_\delta+\delta\partial_x^2 v_\delta
+R_\delta,
\ee
where 
\begin{align*}
\gamma_\delta&=\frac{1}{\pi}\frac{u_\delta}{\delta+u_\delta^2+(Hu_\delta)^2}\quad, \quad	V_\delta=-\frac{1}{\pi}\frac{Hu_\delta}{\delta+u_\delta^2+(Hu_\delta)^2},\\
R_\delta&=-\gamma_\delta(\Lambda e^{(\eps+\delta)\partial_x^2}u_0)-V_\delta(\partial_x e^{(\eps+\delta)\partial_x^2}u_0)+\delta \partial_x^2 e^{(\eps+\delta)\partial_x^2}u_0.
\end{align*}
As already seen in \e{21}, we have
$\gamma_\delta\les_{c_0} 1$ and $\la V_\delta\ra\les_{c_0} 1$. 
By combining this with the fact that 
$e^{\eps\partial_x^2}$ is a smoothing operator, 
we immediately see that
$$
\lA R_\delta\rA_{L^\infty([0,T];L^2)}\les_{c_0,\eps}\lA u_0\rA_{\dot{H}^\mez}.
$$
Here the implicit constant depends on $\eps$, but this is harmless since $\eps$ is fixed now. On the other, directly from \e{z1b}, we get that
$$
\lA \gamma_\delta\Lambda v_\delta\rA_{L^2([0,T];L^2)}
\les_{c_0}\lA \sqrt{\gamma_\delta}\Lambda v_\delta\rA_{L^2([0,T];L^2)}\les_{c_0}M,
$$
and 
$$
\delta\lA \partial_x^2 v_\delta\rA_{L^2([0,T];H^{-\mez})}\le \sqrt{\delta}\lA \partial_x^2 v_\delta\rA_{L^2([0,T];H^{-\mez})}
\le M.
$$
It remains only to estimate the contribution of $V_\delta\partial_x v_\delta$. 
For this, we begin by proving that $(v_\delta)_{\delta\in (0,1]}$ is bounded in $L^p([0,T];\dot{H}^1(\xS))$ for any $1\le p<2$. Indeed, we can write
\begin{align*}
\lA v_{\delta}\rA_{\dot{H}^1}^2&\le  \lA\frac{u_{\delta}^2+(Hu_{\delta})^2}{u_{\delta}}
\rA_{L^\infty}
\int_\xS \frac{u_{\delta} \la \Lambda v_{\delta}\ra^2}{u_{\delta}^2+(Hu_{\delta})^2}\dx\\
&\les_{c_0}\lA (u_{\delta},Hu_{\delta})\rA_{L^\infty}^2
\int_\xS \frac{u_{\delta} \la \Lambda v\ra^2}{u_{\delta}^2+(Hu_{\delta})^2}\dx
\\
&\lesssim_{c_0} \Big(\lA (v_{\delta},Hv_{\delta})\rA_{L^\infty}^2+\eps^{-\mez}\lA u_0\rA_{H^{\mez}}^2\Big)
\int_\xS \frac{u_{\delta} \la \Lambda v\ra^2}{u_{\delta}^2+(Hu_{\delta})^2}\dx
\\
&\lesssim_{c_0} \Big(\lA v_{\delta}\rA_{\dot{H}^{\mez}}^2\log\big(2+\lA v_{\delta}\rA_{\dot{H}^1}\big)+\varepsilon^{-\mez}
\lA u_0\rA_{H^{\mez}}^2\Big)
\int_\xS \frac{u_{\delta} \la \Lambda v_{\delta}\ra^2}{u_{\delta}^2+(Hu_{\delta})^2}\dx,
\end{align*}
to conclude that
\begin{align*}
\frac{\lA v_{\delta}\rA_{\dot{H}^1}^2}{\log(2+\lA v_{\delta}\rA_{\dot{H}^1})}
\lesssim_{c_0,\varepsilon,\lA u_0\rA_{H^{\mez}}}
\int_\xS \frac{u_{\delta} |\Lambda v_{\delta}|^2}{u_{\delta}^2+(Hu_{\delta})^2}\dx.
\end{align*}
Remembering that $\la V_\delta\ra\les_{c_0} 1$, it immediately follows that, for any $p\in [1,2)$,
$$
\lA V_\delta\partial_x v_\delta\rA_{L^p([0,T];L^2)}
\les M.
$$

Now, by combining all the previous estimates, we see that
$(v_\delta)_{\delta\in (0,1]}$ is bounded in the space
$$
X_p=\left\{ u\in C^0([0,T];H^\mez(\xS)\cap L^p([0,T];H^1(\xS))\,;\, \partial_t u\in 
L^p([0,T];H^{-\mez}(\xS))\right\}.
$$
Since $H^\mez(\xS)$ (resp.\ $H^1(\xS)$) 
is compactly embedded into $H^s(\xS)$ (resp.\ $H^{\mez+s}(\xS)$ 
for any $s<1/2$, 
By the classical Aubin-Lions lemma, this in turn implies that one extract a sequence $(u_{\delta_n})_{n\in \xN}$ which converges strongly in
$$
C^0([0,T];H^s(\xS)\cap L^p([0,T];H^{\mez+s}(\xS)).
$$
Then it is elementary to pass to the limit in the equation.

\section{Smoothing effect}\label{S4}
The goal of this section is to prove the second statement in Theorem~\ref{T1} which asserts that the solution are smooth. 
By classical methods for parabolic equations (see \cite[Chapter 15]{Taylor3}), 
it is easy to prove that solutions which are smooth enough (say with initial data in $H^2(\xS)$) are $C^\infty$ for positive time. So it is sufficient to prove that the solution are at least $H^2$ for positive times. This is the purpose of the following
\begin{proposition} The solution $u$ constructs in the previous  section is such that, for any $\xi>4$, there exists a constant $C=C(\xi)$ such that 
\be\label{z2}
\sup_{t\in [0,T]}	t^{\xi}\lA u(t)\rA_{\dot H^{\frac52}}^2 <+\infty.
\ee
\end{proposition}
\begin{proof}
Since $u$ was constructed as the limit of smooth solutions 
(see \S\ref{S:approximate} and \S\ref{S:compactness}), we will 
prove only {\em a priori} estimates. 

As in the previous part, we work with the function
$$
v=u-e^{\eps\partial_x^2}u_0,
$$
with $\eps$ small enough. 
Recall from \e{4} that $v$ solves
\be\label{4'}
\partial_tv+V\partial_x v+\gamma\Lambda v
=R_\eps(u,u_{0})
\ee
where 
\begin{equation*}
	R_\eps(u,u_{0})=-\gamma\Lambda u_{0,\eps}-V\partial_x u_{0,\eps}.
\end{equation*}

We will estimate the $\dot{H}^2$-norm of $v$. 
For this introduce 
$\tilde{v}=\partial_x^2v$, solution to 
\be\label{tilde4}
\partial_t\tilde{v}+\gamma\Lambda \tilde{v}=V\big(H_{\kappa,1}\Lambda \tilde{v}\big)
+R_0+R_1
\ee
where 
\begin{align*}
	&R_0=\partial_x^2\left(V\left(H_{\kappa,2}\Lambda v\right)\right)+\partial_x^2R_\eps(u,u_{0})\\&
	R_1=-(\partial_x^2\gamma)(\Lambda v)-2(\partial_x\gamma)(\Lambda \partial_xv)+(\partial_x^2V)(H_{\kappa,1}\Lambda v)+2(\partial_x(VH_{\kappa,1}\Lambda \partial_xv)).
\end{align*}
Now, we multiply \e{tilde4} by $\Lambda \tilde{v}$ 
and then integrate in $x$ over $\mathbb{S}$, to obtain
\begin{align*}
\mez\fract \lA \tilde{v}\rA_{\dot H^{\frac{1}{2}}}^2
+\int_\xS \gamma \bla \Lambda\tilde{v}\bra^2 \dx 
&=\frac{1}{2}\int_\xS [V,H_{\kappa,1}](\Lambda \tilde{v}) \Lambda \tilde{v} \dx \\
&\quad+\int_\xS R_1\Lambda \tilde{v} \dx
+\int_\xS R_0\Lambda \tilde{v} \dx.
\end{align*}
Set
$$
\tilde{W}\defn \sqrt{\gamma}\Lambda\tilde{v}.
$$

By using arguments parallel to those used in the first step of the proof of Lemma~\ref{L:3.2}, one finds that
\begin{align*}
\left|\int [V,H_{\kappa,1}](\Lambda \tilde{v}) \Lambda \tilde{v} \ dx \right|\lesssim_{c_0}	\lA v\rA_{\dot H^{1/2}}(1+\lA v\rA_{\dot H^{1/2}})^3||\tilde{w}||_{L^2}^2+	C(\varepsilon)\kappa^{\frac{1}{4}}(\lA v\rA_{\dot H^{1/2}}+1)^2 ||\tilde{w}| |_{L^{2}}^2
\end{align*}
and, 
\begin{align*}
	|\int R_0\Lambda \tilde{v} \ dx|\lesssim_{c_0} C(\varepsilon,\kappa)||\tilde{w}||_{L^2}^{\frac{9}{5}}+C(\varepsilon,\kappa).
\end{align*}
Since
\begin{align*}
&	|\partial_x \gamma|+|\partial_xV|\lesssim \frac{1}{\rho^2}\left(|\partial_x u|+|\Lambda u|\right),\\&
	|\partial_x^2\gamma|+|\partial_x^2V|\lesssim \frac{1}{\rho^2}\left(|\partial_x^2 u|+|H\partial_x^2 u|\right)+\frac{1}{V^{3/2}}\left(|\partial_x u|^2+|\Lambda u|^2\right),
\end{align*}
then one finds 
\begin{align*}
|	R_1|&\lesssim_{c_0} \left( \frac{1}{\rho^2}\left(|\partial_x^2 u|+|H\partial_x^2 u|\right)+\frac{1}{\rho^{3}}\left(|\partial_x u|^2+|\Lambda u|^2\right)\right)\left(\la \Lambda v\ra+|H_{\kappa,1}\Lambda v|\right)\\&+\frac{1}{\rho^2}\left(|\partial_x u|+|\Lambda u|\right)\left(|\Lambda \partial_xv|+|H_{\kappa,1}\Lambda \partial_xv|\right).
\end{align*}
We obtain,
\begin{align*}
	|	R_1|&\lesssim_{c_0} \left( \frac{1}{\rho^2}\left(|\partial_x^2 v|+|H\partial_x^2 v|\right)+\frac{1}{\rho^{3}}\left(|\partial_x v|^2+\la \Lambda v\ra^2\right)\right)\left(\la \Lambda v\ra+|H_{\kappa,1}\Lambda v|\right)\\&+\frac{1}{\rho^2}\left(|\partial_x v|+\la \Lambda v\ra\right)\left(|\Lambda \partial_xv|+|H_{\kappa,1}\Lambda \partial_xv|\right)+C(\varepsilon)\left(\la \Lambda v\ra+|H_{\kappa,1}\Lambda v|+|\Lambda \partial_xv|+|H_{\kappa,1}\Lambda \partial_xv|\right)
\end{align*}
Thus, 
\begin{align*}
	\left|\int R_1\Lambda \tilde{v} \ dx \right|\lesssim_{c_0} 	\lA v\rA_{\dot H^{1/2}}(1+\lA v\rA_{\dot H^{1/2}})^3||\tilde{w}||_{L^2}^2+	C(\varepsilon)\kappa^{\frac{1}{4}}(\lA v\rA_{\dot H^{1/2}}+1)^{4} ||\tilde{w}||_{L^{2}}^2+C(\kappa,\varepsilon)
\end{align*}
Therefore, we find
\begin{align*}
	\partial_t ||\tilde{v}||_{\dot H^{\frac{1}{2}}}^2+\int |\tilde{w}|^2 \ dx &\lesssim_{c_0}	\lA v\rA_{\dot H^{1/2}}(1+\lA v\rA_{\dot H^{1/2}})^3||\tilde{w}||_{L^2}^2+	C(\varepsilon)\kappa^{\frac{1}{4}}(\lA v\rA_{\dot H^{1/2}}+1)^4||\tilde{w}||_{L^{2}}^2+C(\kappa,\varepsilon)
\end{align*}
Choosing $\varepsilon$ and then $\kappa$ small enough, we obtain 
\begin{align*}
	\partial_t ||\tilde{v}||_{\dot H^{\frac{1}{2}}}^2+\int_{\mathbb{S}} |\tilde{w}|^2 \ dx &\lesssim C(c_0,\kappa,\varepsilon),
\end{align*}
for any $t\in (0,T)$. So, integrating in time $\tau \in (s,t)$ we have obtained that,
\begin{align*}
||v(t)||_{\dot H^{\frac{5}{2}}}^2+\int_s^t\int_{\mathbb{S}} \frac{u |\Lambda^3v|^2}{u^2+(Hu)^2} \ dx \ d\tau &\lesssim  C(c_0,\kappa,\varepsilon)+||v(s)||_{\dot H^{\frac{5}{2}}}^2.
\end{align*}

Then, in order to measure the decay rate in time, we multiply the last inequality by $s^{\xi-1}$ and then integrate in $s \in [0,T]$, one finds
\begin{align*}
	 \int_{0}^{T} s^{\xi-1}||v(t)||_{\dot H^{\frac{5}{2}}}^2 \ ds + \int_{0}^{T} s^{\xi-1}\int_s^t\int_{\mathbb{S}} \frac{u |\Lambda^3v|^2}{u^2+(Hu)^2} \ dx \ d\tau \ ds &\lesssim \int_{0}^T s^{\xi-1}C(c_0,\kappa,\varepsilon)+s^{\xi-1}||v(s)||_{\dot H^{\frac{5}{2}}}^2 \ ds.
\end{align*}
Then, 
\begin{align*}
	 \sup_{0<t<T} t^{\xi}||v(t)||_{\dot H^{\frac{5}{2}}}^2+\int_0^ts^\xi\int \frac{u |\Lambda^3v|^2}{u^2+(Hu)^2}&\lesssim T^{\xi} C(c_0,\kappa,\varepsilon)+\int_0^ts^{\xi-1}||v(s)||_{\dot H^{\frac{5}{2}}}^2 \ ds.
\end{align*}
Since 
\begin{align*}
\lA v\rA_{H^3}^2&\lesssim_{c_0} \lA (u,Hu)\rA_{L^\infty}^2
\int \frac{u |\Lambda^3v|^2}{u^2+(Hu)^2}\\&\lesssim_{c_0} (\lA (v,Hv)\rA_{L^\infty}^2+\varepsilon^{-1/2}\lA u_0\rA_{H^{1/2}}^2)
\int \frac{u |\Lambda^3v|^2}{u^2+(Hu)^2} \ dx
\\&\lesssim_{c_0} (\lA v\rA_{H^{1/2}}^2\log(2+\lA v\rA_{H^3})+\varepsilon^{-1/2}\lA u_0\rA_{H^{1/2}}^2)
\int \frac{u |\Lambda^3v|^2}{u^2+(Hu)^2} \ dx.
\end{align*}
Then, using the inequality
\begin{align*}
\frac{\lA v\rA_{H^3}^2}{\log(2+\lA v\rA_{H^3})}\lesssim_{c_0,\varepsilon,\lA u_0\rA_{H^{1/2}}}
\int \frac{u |\Lambda^3v|^2}{u^2+(Hu)^2},
\end{align*}
we find that,
\begin{align*}
	\sup_{0<t<T} t^{\xi}||v(t)||_{\dot H^{\frac{5}{2}}}^2+\int_0^ts^\xi\frac{\lA v\rA_{H^3}^2}{\log(2+\lA v\rA_{H^3})} &\lesssim t^\xi C(c_0,\kappa,\varepsilon)+\int_0^ts^{\xi-1}||v(s)||_{\dot H^{\frac{5}{2}}}^2.
\end{align*}
Then, we use the fact that some $\delta>0$ small enough, $a^{2-\delta} \lesssim \frac{a^2}{log(2+a)},$ 
we find
\begin{align} \label{ine}
	 \sup_{0<t<T}t^{\xi}||v(t)||_{\dot H^{\frac{5}{2}}}^2+\int_0^T s^\xi \Vert u(s) \Vert^{2-\delta}_{\dot H^{3}} &\lesssim T^{\xi+1}+ \int_0^T s^{\xi-1} \Vert u(s) \Vert^{2}_{\dot H^{5/2}} \ ds .
\end{align}

Let $\epsilon>0$ and $\delta_{0}>0$, using the following interpolation inequality 
\begin{align*}
s^{\xi-1}\Vert u(s) \Vert^{2}_{\dot H^{5/2}} \leq \epsilon s^{\xi-1+\delta_{0}}\Vert u(s) \Vert^{8/5}_{\dot H^{3}}\ \epsilon^{-1}s^{-\delta_{0}}\Vert u(s) \Vert^{2/5}_{\dot H^{1/2}}, 
\end{align*}
together with Young's inequality (with the conjugate exponents $p=\frac{10-5\delta}{8}$ and $q=\frac{10-5\delta}{2-5\delta}$) we find
\begin{align*}
\int_{0}^{T} s^{\xi-1}\Vert u(s) \Vert^{2}_{\dot H^{5/2}} \ ds &\leq \frac{\epsilon^p}{p}\int_{0}^{T} s^{p\xi-p+p\delta_{0}} \Vert u(s) \Vert^{2-\delta}_{\dot H^{3}} \ ds \\&+ \frac{1}{q\epsilon^{q}} \sup_{s\in (0,T)}\Vert u(s) \Vert^{2q/5}_{\dot H^{1/2}} \int_{0}^{T} s^{-q\delta_0}  \ ds.
\end{align*}
Therefore, inequality \eqref{ine} becomes
\begin{align} \label{ine2}
	 \sup_{0<t<T}t^{\xi}||v(t)||_{\dot H^{\frac{5}{2}}}^2+\int_0^T s^\xi \Vert u(s) \Vert^{2-\delta}_{\dot H^{3}} \ ds &\lesssim& \nonumber T^{\xi+1}+ \frac{\epsilon^p}{p}\int_{0}^{T} s^{p\xi-p+p\delta_{0}} \Vert u(s) \Vert^{2-\delta}_{\dot H^{3}} \ ds \\
	 && \ + \  \frac{1}{q\epsilon^{q}} \sup_{s\in (0,T)}\Vert u(s) \Vert^{2q/5}_{\dot H^{1/2}} \int_{0}^{T} s^{-q\delta_0}  \ ds.
\end{align}

Choosing $m,\delta_0$ such that $q\delta_0<1$ (so that the last integral in the right hand side is finite) and, for a scaling purpose, we also need that
\begin{align*}
\xi=p\xi-p+p\delta_{0}.
\end{align*}
So, 
\begin{equation*}
\xi=q-\delta_0q>q-1.
\end{equation*}
Then, for any $\xi>4$ (note that $q>4$ for any $\delta>0$) and  any $\epsilon >0$ sufficiently small, we may absorb the $\dot H^3$ in the left hand side of inequality \eqref{ine2}, we obtain
\begin{align*} 
	 \sup_{0<t<T}t^{\xi}||v(t)||_{\dot H^{\frac{5}{2}}}^2+\int_0^T s^\xi \Vert u(s) \Vert^{2-\delta}_{\dot H^{3}}  ds\lesssim 1.
\end{align*}
In particular, for any $\xi>4$
\begin{align*} 
	 \sup_{0<t<T}t^{\xi}||v(t)||_{\dot H^{\frac{5}{2}}}^2\leq C.
\end{align*}
\end{proof}

\section{Uniqueness}\label{S5}
The goal of this section is to prove Theorem \ref{Tu}.  For any data 
$u_0 \in H^{1/2+a}(\mathbb{S})$ we know from Theorem \ref{T1} that $u$ in $ C^0([0,+\infty);\dot{H}^{\frac{1+a}{2}}(\xS))\cap L^2((0,+\infty);\dot{H}^{1+\frac{a}{2}}(\xS))$ and satisfy the following condition 
$$\displaystyle \sup_{t>0} t^{s+\varepsilon_0}|| u(t)||_{\dot H^{\frac{1+a}{2}+s}}<\infty$$ for any $s>0.$
	 We will follow an idea in \cite{CNX-2021}. 
 Let $u_1, u_2$ be two solutions of \eqref{1} with same initial data.  Assume that $u_1,u_2\geq c_0>0$. Set $\overline{u}=u_1-u_2$. Let $x_t, x^t$ satisfy  $\overline{u}(x_t,t) =\displaystyle\sup_{x} \overline{u}(x,t)$ and  $-\overline{u}(x^t,t) =\displaystyle\sup_{x} (-\overline{u})(x,t)$.  By evaluating the evolution equation at $x=x_t$, one finds
 \begin{align*}
 &	\partial_t(\overline{u}(x_t,t))+\frac{1}{\pi}\frac{u_1}{u_1^2+(Hu_1)^2}\ \Lambda \overline{u}\\&=-\frac{1}{\pi }\left(\frac{u_1}{u_1^2+(Hu_1)^2}-\frac{u_2}{u_2^2+(Hu_2)^2}\right)\Lambda u_2+\frac{1}{\pi }\left(\frac{Hu_1}{u_1^2+(Hu_1)^2}-\frac{Hu_2}{u_2^2+(Hu_2)^2}\right)\partial_{x} u_2
\\&\lesssim_{c_0}\left( ||\overline{u}||_{L^\infty_x}+|H\overline{u}(x_t,t)|\right) ||(\Lambda u_2,\partial_{x} u_2)||_{L^\infty_x}.
 \end{align*}
We need to control $|H\overline{u}(x_t,t)|$ by $\Lambda \overline{u}(x_t,t)$ and $||\overline{u}||_{L^\infty_x}$.  For any $\varepsilon\in (0,1/2)$, since $\overline{u}(x_t,t)\geq 0$, we may write 
\begin{align*}
	|H\overline{u}(x_t,t)|\lesssim |\log(\varepsilon)|||\overline{u}||_{L^\infty_x}+\int_{|\alpha|\leq\varepsilon}\delta_\alpha \overline{u}(x_t,t)\frac{d\alpha}{|\alpha|}\lesssim |\log(\varepsilon)|||\overline{u}||_{L^\infty_x}+\varepsilon |\Lambda \overline{u}(x_t,t)|
\end{align*}
 Thus, for any $\varepsilon\in (0,1/2)$
\begin{align*}
	&	\partial_t(\overline{u}(x_t,t))+\frac{1}{\pi}\frac{u_1}{u_1^2+(Hu_1)^2}\ \Lambda \overline{u}\\ &\lesssim_{c_0}  |\log(\varepsilon)| ||(\Lambda u_2,\Lambda u_2)||_{L^\infty_x} ||\overline{u}||_{L^\infty_x}+\varepsilon ||(\Lambda u_2,\partial_{x} u_2)||_{L^\infty_x} \Lambda \overline{u}\\ &  \lesssim_{c_0}   |\log(\varepsilon)| ||(\Lambda u_2,\Lambda u_2)||_{L^\infty_x} ||\overline{u}||_{L^\infty_x}+ \varepsilon (1+||u_1||_{ L^\infty\dot  H^{\frac{1+a}{2}}}) ||(\Lambda u_2,\partial_{x} u_2)||_{L^\infty_x} \frac{u_1}{u_1^2+(Hu_1)^2}\ \Lambda \overline{u}.
\end{align*}
Letting $\varepsilon>0$
\begin{equation*}
	\varepsilon (1+||u_1||_{ L^\infty\dot  H^{\frac{1+a}{2}}}) ||(\Lambda u_2,\partial_{x} u_2)||_{L^\infty_x} \ll 1,
\end{equation*}
thus one gets,
$$
	\partial_t(\overline{u}(x_t,t))\lesssim_{c_0}   \log\left((2+||u_1||_{\dot L^\infty H^{\frac{1+a}{2}}})(1+ ||(\Lambda u_2,\partial_{x} u_2)||_{L^\infty_x})\right) ||(\Lambda u_2,\partial_{x} u_2)||_{L^\infty_x} ||\overline{u}||_{L^\infty_x}.$$
Using similar arguments, it is easy to get that
$$
\partial_t(\overline{u}(x^t,t))\lesssim_{c_0}   \log\left((2+||u_1||_{\dot L^\infty H^{\frac{1+a}{2}}})(1+ ||(\Lambda u_2,\partial_{x} u_2)||_{L^\infty_x})\right) ||(\Lambda u_2,\partial_{x} u_2)||_{L^\infty_x} ||\overline{u}||_{L^\infty}.$$
Thus, we obtain, 
\begin{align*}
	||\overline{u}(t)||_{L^\infty}\lesssim 	||\overline{u}(0)||_{L^\infty}+\int_{0}^{t}\omega(\tau) ||\overline{u}(\tau)||_{L^\infty}d\tau,
\end{align*}
where, 
\begin{align*}
	\omega(t)= \log\left((2+||u_1||_{\dot L^\infty H^{\frac{1+a}{2}}})(1+ ||(\Lambda u_2,\partial_{x} u_2)(t)||_{L^\infty_x})\right) ||(\Lambda u_2,\partial_{x} u_2)(t)||_{L^\infty_x} 
\end{align*}
Note that, for any $a_0\in (0,1)$, 
\begin{align*}
	\int_0^t \omega(\tau) d\tau\lesssim_{a_0}  (2+||u_1||_{\dot L^\infty H^{\frac{1+a}{2}}}) \left[1+\sup_{\tau\in (0,t)}\left( \tau^{1-a_0}||(\Lambda u_2,\partial_{x} u_2)(\tau)||_{L^\infty_x} \right)\right]^2.
\end{align*}
 
By Sobolev's inequality for any $a_0\in (0,a/10)$
\begin{align*}
	\sup_{\tau\in (0,t)}\left( \tau^{a_0}||(\Lambda u_2,\partial_{x} u_2)(\tau)||_{L^\infty_x} \right)&\lesssim t ||u_1||_{\dot L^\infty H^{\frac{1}{2}}}+\sup_{\tau\in (0,t)}\left( \tau^{1-a_0}||u_2(\tau)||_{H^{\frac{3}{2}+\frac{a}{4}}} \right)\\&\lesssim C_t<\infty.
\end{align*}
Therefore, by using Gronwall's inequality, one obtains
\begin{align*}
		||\overline{u}(t)||_{L^\infty}\lesssim ||\overline{u}(0)||_{L^\infty}\exp(C(c_0)\int_0^t\omega(\tau) d\tau)\lesssim C_t||\overline{u}(0)||_{L^\infty}.
\end{align*}
Hence, we proved the wellposedness of \eqref{1} in $H^{\frac{1}{2}+a}$ initial data for any $a>0$.

\section*{Acknowledgments}

\noindent  T.A.\ acknowledges the SingFlows project (grant ANR-18-CE40-0027) 
of the French National Research Agency (ANR).  Q-H.N.\ 
is  supported  by the ShanghaiTech University startup fund and the National Natural Science Foundation of China (12050410257).

\vfill

\begin{flushleft}
\textbf{Thomas Alazard}\\
Universit{\'e} Paris-Saclay, ENS Paris-Saclay, CNRS,\\
Centre Borelli UMR9010, avenue des Sciences, 
F-91190 Gif-sur-Yvette\\
France

and

Department of Mathematics, University of California, Berkeley, \\
Berkeley, CA 94720\\
USA

\vspace{1cm}

\textbf{Omar Lazar}\\
College of Engineering and Technology,\\
American University of the Middle East,\\
Kuwait

and

Departamento de An\'alisis Matem\'atico \& IMUS,\\
Universidad de Sevilla,\\
Spain

\vspace{1cm}

\textbf{Quoc-Hung Nguyen}\\
ShanghaiTech University, \\
393 Middle Huaxia Road, Pudong,\\
Shanghai, 201210,\\
China

\end{flushleft}

\end{document}